\documentclass{amsart}
\usepackage{amsmath,amssymb,amsthm}
\usepackage{xy}
\usepackage{graphicx}
\numberwithin{equation}{section}
\newtheorem{theorem}[equation]{Theorem}
\newtheorem{proposition}[equation]{Proposition}
\newtheorem{corollary}[equation]{Corollary}
\newtheorem{lemma}[equation]{Lemma}
\newtheorem*{theorem*}{Theorem}

\theoremstyle{remark}

\newtheorem{example}[equation]{Example}
\newtheorem{remark}[equation]{Remark}

\newtheorem*{remark*}{Remark}

\def\Aut{\mathop{\operator@font Aut}}

\input xy
\xyoption{matrix}
\xyoption{arrow}

\newcommand{\xycenter}[1]{\begin{center}
                          \mbox{\xymatrix{#1}}
                          \end{center}
                         }

\begin{document}

\title[Vector bundles over Kodaira surfaces]{
Two-dimensional moduli spaces of vector bundles over Kodaira surfaces}

\author[APRODU, MORARU \& TOMA]{
Marian Aprodu, Ruxandra MORARU \& Matei TOMA}

\thanks{Marian Aprodu was partly supported by the grant
PN-II-ID-PCE-2011-3-0288 and LEA MathMode. Ruxandra Moraru was partially supported by NSERC.
Matei Toma was partly supported by LEA MathMode.}

\address{{\it Marian Aprodu}: 
Romanian Academy, 
Institute of Mathematics "Simion Stoilow" 
P.O. Box 1-764, RO 014700, 
Bucharest, Romania}
\email{marian.aprodu@imar.ro}
\address{{\it Marian Aprodu}:
\c Scoala Normal\u a Superioar\u a Bucure\c sti, Calea Grivi\c tei 21,
RO-010702, Bucharest, Romania}

\address{{\it Ruxandra Moraru}:
Department of Pure Mathematics,
University of Waterloo,
200 University Avenue West, Waterloo, Ontario
Canada N2L 3G1}
\email{moraru@math.uwaterloo.ca}

\address{{\it Matei Toma}: Institut \'Elie Cartan Nancy, UMR 7502, Universit\'e de Lorraine, CNRS, INRIA, Boulevard des Aiguillettes, B.P. 70239, 54506 Vandoeuvre-l\`es-Nancy Cedex, France }
\email{Matei.Toma@iecn.u-nancy.fr}

\begin{abstract}
We prove that any two-dimensional moduli space of stable 2-vector
bundles, in the non-filtrable range, on a primary Kodaira surface is a 
primary Kodaira surface.
If a universal bundle exists, then the two surfaces are
homeomorphic up to unramified covers. 
\end{abstract}

\maketitle

\section{Introduction}

Non-trivial examples of holomorphic symplectic
manifolds have been obtained from moduli spaces of semi-stable sheaves
over projective holomorphic symplectic surfaces (Mukai, Tyurin, O'Grady,
see \cite{HL} and the references therein). 
By classification, any such surface is either K3 or abelian. 
Absent the projectivity assumption, or even the K\"ahler condition, 
one more class of holomorphic symplectic surfaces appears.
Precisely, these are the primary Kodaira surfaces, which are defined
as topologically non-trivial principal elliptic bundles 
$X\to B$ over an elliptic curve $B$.
It is known that moduli spaces of stable sheaves on
primary Kodaira surfaces inherit the holomorphic symplectic
structure. In some situations, these moduli spaces are
compact \cite{Toma}. 

In this paper, we study two-dimensional moduli spaces $\mathcal M$
of stable rank-2 vector bundles with fixed determinant and fixed second Chern class
on primary Kodaira surfaces; here by ``fixed determinant"
we mean that the determinant line bundle has a fixed
isomorphism type. 
The first observation is that they are compact if
we place ourselves in the {\em non-filtrable range}
(see section \ref{sec:2-dim}); note that this hypothesis is
equivalent to the apparently weaker condition (*) from 
\cite{Toma} which says that every semi-stable vector bundle
with suitable topological invariants is stable.
We show that, for these spaces, quasi-universal sheaves
exist and that they are sometimes even universal sheaves.
Using the spectral cover construction from \cite{Brinzanescu-Moraru1},
we infer that the spaces $\mathcal M$ are principal elliptic bundles over
the same elliptic curve $B$, hence they are either
tori or primary Kodaira surfaces. We perform an
analysis of the torus case, and we note that a
universal family exists after passing to an \'etale cover. We are led to a situation where $X$ may be seen as a parameter space for vector bundles on $\mathcal M$. Using 
 the techniques developed in \cite{Teleman-CCM},
we get a contradiction of the fact that $\mathcal M$ was supposed to be K\"ahler.
The bottom line is that these spaces $\mathcal M$ must be
primary Kodaira surfaces. 

In the last section, we prove that if a universal sheaf
exists, then the moduli space $\mathcal M$ is homeomorphic
to an unramified cover of the original base surface $X$.

\medskip
{\em Acknowledgements.} We would like to thank
the referee for  carefully reading the paper and
for his pertinent remarks, which helped improve
the exposition of the paper.

\section{Rank-2 vector bundles on Kodaira surfaces} \label{general}
\label{sec:rank-2}

For any smooth compact complex surface $X$ and any pair of Chern classes
$(c_1,c_2) \in \mathbb \mathrm{NS}(X) \times \mathbb{Z}$,
one defines the discriminant as:
\[
 \Delta(2,c_1,c_2) := \frac{1}{2}\left(c_2-\frac{1}{4}c_1^2\right).
\]

On non-algebraic surfaces, the intersection form on the Neron-Severi group
is negative semi-definite, and this can be used to show that 
if $(c_1,c_2)$ are the Chern classes of a rank-2 holomorphic vector bundle on $X$,
then $\Delta(2,c_1,c_2)\ge 0$ \cite{BL,BrinzanescuLNM}.
A natural problem is to determine whether or not the non-negativity of
$\Delta(2,c_1,c_2)$ suffices for the existence of a holomorphic
rank-2 vector bundle $E$ on $X$ with $c_1(E)=c_1$ and $c_2(E)=c_2$.
This problem was addressed for all non-algebraic surfaces in \cite{BL}. For primary
Kodaira surfaces, a complete solution 
was given in \cite{ABT}. 

Related to this problem is the notion of filtrability: a rank-2 torsion-free sheaf on 
$X$ is called {\em filtrable} if it has a rank-one coherent subsheaf
(in higher ranks one would ask for a filtration with terms in every rank).
On a non-algebraic surface $X$, filtrable rank-2 vector bundles with
Chern classes $c_1$ and $c_2$ exist if and only if 
\[
\Delta(2,c_1,c_2)\ge m(2,c_1),
\]
where
\[
m(2,c_1):=-\frac{1}{2}\mathrm{max}\left\{\left(\frac{c_1}{2}-\mu\right)^2\ :\
\mu \in \mathrm{NS}(X)\right\}
\]
\cite{BL,BrinzanescuLNM}. In the case where $X$ is a primary Kodaira surface, it is proved in \cite{ABT} that holomorphic rank-2 bundles on $X$ with
Chern classes $c_1$ and $c_2$ exist even in the
{\em non-filtrable range} 
\[0\le \Delta(2,c_1,c_2)< m(2,c_1).\]
In particular, these bundles are stable with respect to any Gauduchon
metric on $X$.

For the rest of this section, we work on primary Kodaira surfaces.
Let $X$ be a primary Kodaira surface and denote its elliptic fibration by 
$\pi : X \rightarrow B$, where $B$ is a genus 1 curve, with elliptic fibre $T$.
Recall that $X$ is isomorphic to a quotient of the form $\Theta^*/\langle \tau \rangle$, where $\Theta$ is a line bundle on $B$ of postive degree $m$, $\Theta^*$ is the complement of the zero-section in the total space of $\Theta$, and $\langle \tau \rangle$ is the multiplicative cyclic group generated by a nonzero complex number $\tau$ with $|\tau| > 1$. The fibres of $\pi:X \rightarrow B$ are then isomorphic to $T = \mathbb{C}^*/\langle \tau \rangle$, and the cohomology of $X$ is given by
\[ H^1(X,\mathbb{Z}) = \mathbb{Z}^{\oplus 3}; H^2(X,\mathbb{Z}) = \mathbb{Z}^{\oplus 4} \oplus \mathbb{Z}/m\mathbb{Z}; H^3(X,\mathbb{Z}) = \mathbb{Z}^{\oplus 3} \oplus \mathbb{Z}/m\mathbb{Z}.\]
Note in particular that these cohomology groups have no torsion if $d=1$; moreover, the torsion of $H^2(X,\mathbb{Z})$, and therefore of $\mathrm{NS}(X)$, is generated by the class of a fibre of $\pi$. We shall see in section \ref{sec:pairs} that primary Kodaira surfaces are completely determined, up to homeomorphism, by $m = \mathrm{ord}(Tors(H^2(X,\mathbb{Z}))) = \mathrm{ord}(Tors(\mathrm{NS}(X)))$. Finally, the canonical bundle of $X$ is trivial and $\chi(\mathcal{O}_X) = 0$ (see \cite {BHPV} for details).

Consider a pair $(c_1,c_2) \in \mathrm{NS}(X) \times \mathbb{Z}$
and a line bundle $\delta$ on $X$ with
$c_1(\delta) = c_1$. Let $\mathcal{M} := \mathcal{M}_{\delta,c_2}$ be the moduli space of rank-2 
{\em torsion free sheaves} on $X$ with determinant $\delta$ and second Chern class $c_2$.
Referring to \cite{Toma}, if the pair $(c_1,c_2)$ is in the non-filtrable range, the moduli space
$\mathcal{M}$ is 
compact, smooth and holomorphic symplectic,
of dimension $8\Delta(2,c_1,c_2)$. For $\Delta(2,c_1,c_2) = 1/4$, the moduli
space $\mathcal M$ is thus a compact holomorphic symplectic surface.

To describe the geometry of $\mathcal{M}$ in sections \ref{sec:2-dim} and \ref{sec:pairs}, we will use
the  notion of {\em spectral cover} associated to a rank-2 vector 
bundle \cite{Brinzanescu-Moraru1}, which we briefly recall.
Let $T^*$ denote the dual of $T$ (we fix a non-canonical identification $T^*:=\mathrm{Pic}^0(T)$).
It is known \cite{BrinzanescuLNM} that the relative Jacobian $J(X)$ of $\pi:X\rightarrow B$
is trivial, isomorphic to $B \times T^*$.
Let us denote by $T_b$ the fibre $\pi^{-1}(b)$ of $\pi$ over a point $b\in B$.
The line bundle $\delta$ then defines a morphism
$B\to T^*$, $b\mapsto \delta_b := \delta|_{T_b}\in \mathrm{Pic}^0(T)$,
which can be viewed as a section of $J(X) = B \times T^* \rightarrow B$, called $\Sigma_\delta$.
Moreover, if $\delta$ and $\delta'$ are line bundles on $X$ with $c_1(\delta) = c_1(\delta')$, then
$\delta = \lambda \otimes \delta'$ for some $\lambda \in \pi^*\mathrm{P}ic(B)$ \cite{Brinzanescu-Moraru1}.
This means in particular that $\delta_b = \delta'_b$ for all $b \in B$ as $\lambda_b = \mathcal{O}_{T_b}$ for all $b \in B$.
The morphism $B \rightarrow T^*, b \mapsto \delta_b$, therefore only depends on the Chern class $c_1$ of $\delta$ and so we denoted by
\[ \hat{c}_1: B \rightarrow T^*;\]
this provides an isomorphism between 
$\mathrm{NS}(X)/Tors(\mathrm{NS}(X))$ and the space of sections
of $J(X)$ \cite{BrinzanescuLNM}:
\[ 
\mathrm{NS}(X)/Tors(\mathrm{NS}(X)) \simeq \mathrm{Hom}(B,T^*).
\]
In particular, if $\mathrm{NS}(X)$ has no torsion and $B = T$, then $\mathrm{NS}(X) \simeq \mathrm{End}(B)$.

Let $E$ be a holomorphic rank-2 vector bundle on $X$ with determinant $\delta$ and second Chern class $c_2$.
One associates to $E$ several geometric objects \cite{Brinzanescu-Moraru1} that are crucial to our analysis. Firstly, there is the 
{\em bisection} of $J(X) \rightarrow B$ determined by $E$. Proposition 3.2 in \cite{Brinzanescu-Moraru1}
shows that, except at possibly finitely many points of $B$, the restriction
of $E$ to a fibre $T_b$ is semi-stable, and is thus $S$-equivalent
to $\lambda_b\oplus \delta_b^{-1}\lambda_b$, for 
some $\lambda_b\in T^*$. Roughly speaking, the {\em bisection}
 $\overline{C}_E$  of $E$ is the curve in $J(X)$ that intersects $T^*$ in
$\lambda_b$ and $\delta_b^{-1}\lambda_b$ for any of these $b\in B$;
note that the bisection comes equipped with a natural 
projection morphism to $B$.
The {\em spectral curve} $S_E$ of $E$ is the effective divisor in $J(X)$ given by
\[ 
S_E := \left(\sum_{i=1}^k \{ b_i \} \times T^*\right) + \overline{C}_E,
\]
where the points $b_i \in B$ correspond to the fibres $T_b$ over which the restriction of $E$
is unstable, counted with multiplicities (see \cite{Brinzanescu-Moraru1},
p. 1670). Lastly, the {\em spectral cover} of $E$
is the natural morphism $S_E\to B$.
It keeps track of  the isomorphism types of all the restrictions
of $E$ to the fibres $T_b$; in fact, when $S_E$ is smooth, one shows that $E$ is completely determined by $\overline{C}_E$ and a line bundle $L$ on $\overline{C}_E$ \cite{Brinzanescu-Moraru2}. Moreover, note that if the bisection $\overline{C}_E$ is irreducible, then so is $E$ (see proposition 3.2 in \cite{Brinzanescu-Moraru3}).

We have a natural involution $i_\delta$ on $J(X)$ 
given by $(b,\lambda) \mapsto (b,\delta_b \lambda^{-1})$.
Consider the ruled surface $\mathbb{F}_{\delta} := J(X)/i_\delta$ over $B$, defined as the 
quotient of $J(X)$ by the involution $i_\delta$, and
let $\eta:J(X) \rightarrow \mathbb{F}_{\delta}$ be the canonical map. (Note that if $\delta'$ is a line bundle on $X$ such that $c_1(\delta') = c_1(\delta)$, then $i_{\delta'} = i_\delta$, implying that $\mathbb{F}_{\delta'} = \mathbb{F}_\delta$.) By construction, the spectral curve $S_E$ 
associated to $E$ is invariant under the involution $i_\delta$ and descends to the quotient $\mathbb{F}_\delta$. 
It can therefore be considered as the pullback via $\eta$ of a divisor on $\mathbb{F}_\delta$ of the form 
\[ 
G_E := \sum_{i=1}^k f_i + A,
\]
where $f_i$ is the fibre of the ruled surface $\mathbb{F}_\delta$ over the point $b_i$ and $A$ is a section of the ruling 
such that $\eta^* A = \overline{C}_E$. The divisor $G_E$ is called the {\em graph} of the bundle $E$.
Denote by $A_0$ the image in $\mathbb{F}_{\delta}$ of the
zero section of $J(X)$, and by $f$ the class of a fibre of $\mathbb{F}_{\delta}\to B$. 
The graph $G_E$ of $E$ is then an element
of a linear system on $\mathbb{F}_{\delta}$ of the form $|A_0 + \mathfrak{b}f|$,
where $\mathfrak{b}f$ is the pullback to $\mathbb{F}_{\delta}$ of a divisor on $B$ 
of degree $c_2$ (see \cite{Brinzanescu-Moraru1}, section 3.2). This means in particular that spectral curves correspond to elements 
of linear systems on $\mathbb{F}_{\delta}$ of the form $|A_0 + \mathfrak{b}f|$.

\section{Universal and quasi-universal families}

Recall from \cite{Mukai} that, given a compact complex manifold (or scheme) $X$ 
and a connected component $\mathcal M$ of the moduli space
of simple sheaves on $X$, a {\em quasi-universal sheaf} is
a sheaf $\mathcal E$ on $X\times\mathcal M$
for which there exists a positive integer $n$ such that, 
for any sheaf $E$ corresponding to a point $y\in \mathcal M$,
$\mathcal E|_{X\times y}\cong E^{\oplus n}$.
Note that if every sheaf in $\mathcal{M}$ is locally free, then $\mathcal E$ is a vector bundle.
Mukai shows that quasi-universal families of simple
sheaves always exist if $X$ is a projective scheme (see \cite{Mukai}, Theorem A.5). 
One ingredient of his proof is Serre's Theorem B.
Here we give a non-algebraic analogue of this result in a situation
close to our purposes; the methods can, however,
be extended to a more general setup. 
We shall use the terminology of \cite{Mukai}, Appendix 2.

A coherent sheaf of rank $r$ is said to be {\em irreducible} if 
it does not contain any proper subsheaf of rank $<r$. For
$r=2$, irreducibility coincides with non-filtrability.
Such sheaves exist on any non-algebraic surface
and they are automatically stable with respect to any 
Gauduchon metric on $X$,
see for example \cite{BrinzanescuLNM}.  It is therefore not necessary to fix a 
Gauduchon metric on $X$ when considering moduli spaces of irreducible sheaves.

\begin{proposition}\label{prop:quasi-universal}
 If $X$ is a non-algebraic 2-torus or a primary Kodaira surface,
the moduli space of irreducible torsion-free sheaves with fixed determinant
on $X$ admits a quasi-universal sheaf.
\end{proposition}

\proof
Let  $\mathcal{M}$ be
the moduli space of irreducible torsion-free sheaves on $X$ of rank $r>1$, fixed
determinant line bundle $\delta$ and second Chern class $c_2$. 
Let $c_1$ be the first Chern class of $\delta$.
Then $\mathcal{M}$ is smooth of dimension 
$2r^2 \Delta(r, c_1, c_2) -(r^2-1)\chi(\mathcal{O}_X)= 
2r^2 \Delta(r, c_1, c_2)$, where $ \Delta(r, c_1, c_2) := \frac{1}{r}(c_2 - \frac{r-1}{2r}c^2_1)$ (see \cite{Toma}, Corollary 3.4). Assume this dimension to be positive,
otherwise the existence of a universal family is trivial. In this
case, for any sheaf $F$ corresponding to 
a point of $\mathcal{M}$, one has $h^0(X, F)=h^2(X,
F)=0$ by stability of $F$, implying that
$h^1(X, F)=r(\Delta(r, c_1, c_2)- c_1^2/(2r^2)) \geq r \Delta(r,c_1, c_2) > 0$.
Consider an open covering $(U_i)_i$ of $\mathcal{M}$ such that a universal family  $\mathcal{E}_i$ 
exists over each $X \times U_i$ 
 and denote by $\pi _i :  X \times U_i \rightarrow U_i$ the projection. 
 Set $V_i := R^1 _* \pi _i (\mathcal{E}_i)$; then these are vector bundles of 
 rank $r(\Delta(r, c_1, c_2)- c_1^2/(2r^2))$ on $U_i$. Now, on $U_{ij}:=U_i \cap U_j$, we
 have isomorphisms 
 $f_{ij}:\mathcal{E}_i |_{X \times U_{ij}} \rightarrow \mathcal{E}_j |_{X \times
 U_{ij}}$ and 
 $\bar{f}_{ij} := R^1 _* \pi _{ij} (f_{ij}): V_i |_{U_{ij}} \rightarrow V_j |_{U_{ij}}$, where 
 $\pi _{ij} :  X \times U_{ij} \rightarrow U_{ij}$  are the projections.
  As in the
 projective case, we can use the
 $f_{ij} \otimes \pi _{ij} ^* ((\bar{f}^{-1}_{ij})^{\vee}): 
 \mathcal{E}_i \otimes \pi _{i} ^* V_i^{\vee} |_{X \times U_{ij}} \rightarrow 
 \mathcal{E}_j \otimes \pi _j ^* V_j^{\vee} |_{X \times U_{ij}}$ to glue the sheaves
 $\mathcal{E}_i \otimes \pi _i ^* V_i^{\vee}$ into a quasi-universal family on 
 $X \times \mathcal{M}$ (for details, see \cite{Mukai}, Theorem A.5).
\endproof 
 
 Note that the condition that the determinant be fixed is not
 essential to the proof of Proposition \ref{prop:quasi-universal}, and we only introduced it to
 be consistent with the framework we place
 ourselves in.
 
 \medskip
 
 As in the projective case, 
 universal families will exist if the greatest common divisor 
 of the numbers $\chi(F \otimes N)$ is one, where $F$ is a sheaf 
 corresponding to a point of
 $\mathcal{M}$ and $N$ runs through all locally free sheaves
 of rank one on $X$ (for details of the projective case, see \cite{Mukai}, Appendix 2). 
 Moreover, as the next example illustrates, non-algebraic surfaces can admit universal families.

 \begin{example}
\label{ex:universal}
 Take $X$ to be a primary Kodaira surface such that $\mathrm{NS}(X)$ has no torsion and is generated by two elements
 $c_1$, $e_1$ such that $c_1^2=-6$, $e_1^2= -2$ and $c_1\cdot e_1= 1$. To see that such a
 surface exists, take $X$ with base and fiber equal to 
 an elliptic curve $E$ with periods $1$ and $(-1+\sqrt{-11})/2$, and assume that it is obtained from a line bundle $\Theta$ on $E$ of degree 1 thus ensuring that $\mathrm{NS}(X)$ has no torsion (see section \ref{sec:rank-2}). Then use the
 isomorphism  
 $\mathrm{NS}(X)=\mathrm{End}(E)$ and compute the last group as in \cite{Ha}, IV.4.19. Let now
 $\delta$ and 
 $\varepsilon$ be line bundles with Chern classes $c_1$ and $e_1$ respectively. Let $\mathcal{M}$ be the
 moduli space  of rank-2 torsion free sheaves  on $X$ with
 determinant $\delta$
 and discriminant $\Delta(2,c_1,c_2)=1/4$. Referring to section \ref{sec:rank-2}, given that $0 \leq \Delta(2,c_1,c_2)  = 1/4 < 3/4 = m(2,c_1)$,
 the pair $(c_1,c_2)$ is in the non-filtrable range, implying that $\mathcal{M}$ is a smooth compact complex surface. 
 Note that every member of $\mathcal M$ is non-filtrable and therefore stable with respect to any Gauduchon metric on $X$. 
 Moreover, since $\chi(F)= -2 $ and $\chi(F \otimes \varepsilon)= -3$ for all members $F$ of $\mathcal{M}$, 
 the moduli space $\mathcal{M}$ is fine.
\end{example}

Finally, in some cases, universal families exist on an \'etale cover.
 
\begin{proposition}
\label{prop:universal sheaf}
Let $X$ be a compact complex manifold endowed with a Gauduchon metric,
and assume that a moduli space $\mathcal M$ of stable sheaves of rank $r$
with fixed determinant on $X$ is a complex torus. Then there exist an \'etale 
(connected) cover $Y\to \mathcal M$
and a universal sheaf $\mathcal E$ on $X\times Y$.
\end{proposition}
 
\proof
On an open cover $(U_i)_i$ of
$\mathcal M$, there are universal bundles $\mathcal E_i\to X\times U_i$.
By the universality, we obtain isomorphisms
$f_{ij}:\mathcal E_i|_{X\times U_{ij}}\to \mathcal E_j|_{X\times U_{ij}}$.
Over each point $y\in U_{ijk}:=U_i \cap U_j \cap U_k$, the restriction
$\varphi_{ijk}:=(f_{ij}\circ f_{jk}\circ f_{ki})|_{X\times y}$
is a non-trivial endomorphism of $\mathcal E_i|_{X\times y}$
and thus a multiple of the identity, by stability of $\mathcal E_i|_{X\times y}$. 
Varying $y$, we obtain a cocycle $\alpha\in H^2(\mathcal M, \mathcal O^*_{\mathcal M})$.
Let  $\widetilde{\alpha}$ be the pull-back of $\alpha$ via 
the projection $X\times\mathcal M\to \mathcal M$.
The class $\widetilde{\alpha}$ is the obstruction
to glueing the local universal sheaves $\mathcal E_i$ to
a globally defined universal sheaf; note however that the
$\mathcal E_i$ can be glued together to a globally defined universal $\widetilde{\alpha}$-sheaf.
On the other hand, the class $\widetilde{\alpha}$ is a $r$-torsion element, 
as it comes from $H^2(X\times \mathcal M, \mathbb Z/r\mathbb Z)$
via the natural morphism induced by $\mathbb Z/r\mathbb Z\hookrightarrow
\mathcal O_{X\times \mathcal M}^*$ (see \cite{Caldararu-TEZA}, p. 9).
Since the pull-back 
$H^2(\mathcal M,\mathcal O^*_{\mathcal M})\to
H^2(X\times \mathcal M,\mathcal O^*_{X\times \mathcal M})$
is injective, it follows that $\alpha$ is $r$-torsion too.
We pass to an \'etale cover $f:Y\to\mathcal M$ such that
the pull-back of $\alpha$ vanishes. By Proposition 1.2.10 in \cite{Caldararu-TEZA}, the pullback to $X \times Y$ of the universal 
$\widetilde{\alpha}$-sheaf on $X\times \mathcal M$ is an $(\mathrm{id}_X \times f)^*\widetilde{\alpha}$-sheaf
on $X\times Y$, implying in particular that it is a (untwisted) universal sheaf on $X \times Y$ since $(\mathrm{id}_X \times f)^*\widetilde{\alpha}$ vanishes.
\endproof

Again, the hypothesis that the determinant be fixed has not been
used in the proof.

\section{Two-dimensional moduli spaces}
\label{sec:2-dim}

We use the notation of the previous sections.
In the sequel, we work under the assumptions that 
$$
\Delta(2,c_1,c_2) = 1/4,\ m(2,c_1)>1/4,
$$ 
meaning that the moduli space $\mathcal M$ is
two-dimensional and we are in the
non-filtrable range. Note that
\[ -c_1^2/2 = 1 - 2c_2, \]
implying, in particular, that $-c_1^2/2$ is odd. 
In addition, by Lemma 3.8 in \cite{Brinzanescu-Moraru1},
the ruled surface $\mathbb{F}_{\delta}$ has invariant $e = -1$.

\begin{lemma}
Every sheaf in $\mathcal{M}$ is locally free. 
\end{lemma}
\begin{proof}
Let $\mathcal{E} \in \mathcal{M}$. Then, if the skyscraper sheaf 
$\mathcal{E}^{\ast\ast}/\mathcal{E}$ is supported on $m$ points (counting multiplicity),
we have
\[ 8 \Delta(\mathcal{E}^{\ast\ast}) = 8\Delta(\mathcal{E}) - 4m = 2 - 4m \geq 0.\]
Consequently, $0 \leq m \leq 1/2$, implying that $m = 0$ and $\mathcal{E}$ is
locally free.
\end{proof}

\begin{lemma}
Let $E \in \mathcal{M}$. Then $E$ is semi-stable on every fibre of $\pi$. 
In particular, the spectral curve of any vector bundle in 
$\mathcal{M}$ coincides with the associated bisection.
\end{lemma}
\begin{proof}
We prove the lemma by contradiction.
Let us assume, on the contrary, that $E$ is not semi-stable on the fibre 
$\pi^{-1}(x_0) := T_{0}$. Its restriction to $T_{0}$ is therefore of the form
$\lambda \oplus (\lambda^\ast \otimes \delta)$ with 
$\lambda \in {\rm Pic}^{-k}(T_{0})$, $k > 0$.
In addition, up to a multiple of the identity, there is a unique
surjection $E|_{T_{0}} \rightarrow \lambda$, which defines a canonical elementary modification of $E$ that we
denote $\bar{E}$. This elementary modification is called allowable and it has the following invariants (see section 5.1 of \cite{Brinzanescu-Moraru2}): 
\[c_1(\bar{E}) = c_1(E) + c_1(\mathcal{O}_X(-T_{0}))\] 
and 
\[ c_2(\bar{E}) = c_2(E) + c_1(E) \cdot c_1(\mathcal{O}_X(-T_{0})) + c_1(\lambda). \]
Since the fibres of $\pi$ generate the torsion part of $\mathrm{NS}(X)$
as mentioned above, we have $c_1(E) \cdot c_1(\mathcal{O}_X(-T_{0}))=0$ 
and $c_1^2(\mathcal{O}_X(-T_{0})) = 0$.
Consequently,
\[ c_2(\bar{E}) = c_2(E) + c_1(\lambda)  = c_2(E) - k\]
and
\[ \Delta(\bar{E}) = \Delta(E) - \frac{k}{2} = \frac{1}{4} - \frac{k}{2} < 0\]
because $k > 0$. However, a locally free sheaf on $X$ cannot have a negative 
discriminant, leading to a contraction.
\end{proof}

To describe the elements of $\mathcal{M}$ we therefore have to first classify
their possible spectral curves,
which correspond to elements of a linear system in $\mathbb{F}_\delta$ of the form 
$|A_0 + \mathfrak{b}f|$
with $\mathfrak{b}f$ the pullback of a divisor $\mathfrak{b}$ on $B$ of degree $c_2$;
we begin by describing such linear systems. Note that different choices of
$\mathfrak{b}$ give rise to different bisections of $J(X) \rightarrow B$. 
We have the following:

\begin{lemma}
Let $\mathfrak{b}$ be a divisor of degree $c_2$ on $B$.
We denote by $\mathfrak{b}f$ its pullback to $\mathbb{F}_{\delta}$.
Then the linear system $|A_0 + \mathfrak{b}f|$ consists of a single element.
\end{lemma}
\begin{proof}
In this proof, we use the notation of section 3.2 in 
\cite{Brinzanescu-Moraru1}.
Let $p_1: J(X) = B \times T^* \rightarrow B$ denote the projection onto the first factor.
Also, denote by $B_0$ the zero section of $J(X)$ and by $\Sigma_\delta$ 
the section of $J(X)$ corresponding to $\delta$.
Furthermore, let $V_\delta = {p_1}_\ast(\mathcal{O}_{J(X)}(B_0 + \Sigma_\delta))$.
Then $\deg V_\delta = -c_1^2/2$ 
and $\mathbb{F}_{\delta} = \mathbb{P}(V_\delta)$
(see \cite{Brinzanescu-Moraru1}, p. 1673).
Moreover, we have
\[H^0(\mathbb{F}_{\delta},\mathcal{O}_{\mathbb{F}_{\delta}}(A_0 + \mathfrak{b}f)) 
= H^0(B, \mathcal{O}_B(\mathfrak{b}) \otimes V_\delta). \]

Note that since $\mathbb{F}_{\delta}$ has invariant $e = -1$, we have
$V_\delta = \mathfrak a \otimes W$ with $W$ a non-trivial extension of 
line bundles on $B$ of the form
\[ 0 \rightarrow \mathcal{O} \rightarrow W \rightarrow \mathcal{O}(p) \rightarrow 0,\]
for some point $p$ on $B$,
and $\mathfrak a$ the line bundle of maximal degree mapping into $V_\delta$.
Put $\deg \mathfrak a = d$. Then
\[ e = 2d + c_1^2/2,\] 
implying that
\[ -1 = 2d + c_1^2/2 = 2d - (1 - 2c_2).\]

Consequently, $V_\delta $ is given by the non-trivial extension
\[ 0 \rightarrow \mathfrak a \rightarrow V_\delta \rightarrow \mathfrak a(p) \rightarrow 0, \]
with $\mathfrak a$ such that $\deg \mathfrak a = -c_2$.
In particular, this implies that for any divisor $\mathfrak{b}$ on $B$ of degree
$c_2$, we have
\[ H^0(B,\mathcal{O}_B(\mathfrak{b}) \otimes V_\delta )  = \mathbb{C},\]
proving that $|A_0 + \mathfrak{b}f|$ consists of a single point.
\end{proof}

Hence, the set of all possible graphs or, equivalently, 
spectral curves of bundles in
$\mathcal{M}$ is parametrised by ${\rm Pic}^{c_2}(B)$. 

\begin{corollary}
\label{cor:smooth  bisection}
The spectral curve of any vector bundle in $\mathcal{M}$ is smooth and irreducible. 
\end{corollary}
\begin{proof}
Consider the bisection $\overline{C}_E$ of $E\in\mathcal M$,
and $G_E\in|A_0 + \mathfrak{b}f|$ its graph in $\mathbb F_\delta$.
An intersection computation on $\mathbb F_\delta$
shows exactly as in Lemma 3.10 of \cite{Brinzanescu-Moraru1}
(which treats the smooth case) that the arithmetic
genus of $\overline{C}_E$ equals 2. Its normalization
$C$ is thus either a genus-2 curve and $C=\overline{C}_E$, or
an elliptic curve. In the latter case, the composition 
$C\to \overline{C}_E\to B$ is an \'etale cover
and it was shown in \cite{AT} that $E$ is 
an elementary modification of a push-forward of a line
bundle on the fibred-product $C\times_BX$.
However, the discriminant of a push-forward is zero,
and any of its elementary modifications will have a
discriminant of type $n/2$ with $n\in \mathbb N$.
Therefore, the only case that can occur is $C=\overline{C}_E$.
\end{proof}

\begin{remark}
 It follows from Corollary \ref{cor:smooth  bisection} above and
Theorem 4.5 in \cite{Brinzanescu-Moraru2} 
that any rank-2 vector bundle in $\mathcal M$ is
the direct image of a line bundle under a double cover of $X$.
A similar statement is known to be true for $\Delta=0$
\cite{AT}.
\end{remark}

\begin{remark}
\label{rmk: splitting type}
Given any spectral cover, one can construct a bundle in $\mathcal{M}$
corresponding to it (see Theorem 4.4 in \cite{Brinzanescu-Moraru1}). Consequently,
for any given fibre $T_b$ over a point $b\in B$
and any $\lambda\in T^*$, there exists a bundle $E$ on $X$ with splitting type $(\lambda, \delta_b\lambda^{-1})$
over $T_b$. 
\end{remark}

\begin{remark}
Note that the spectral curves $C_E$ are smooth curves of genus 2
(see Lemma 3.6 in \cite{Brinzanescu-Moraru1}).
 Moreover, they are all isomorphic, as curves. Indeed, we have seen that
the spectral covers are parametrised by $\mathrm{Pic}^{c_2}(B)$ and it
is known \cite{Kas} that the moduli space $\mathfrak M_2$ of curves of genus 2
does not contain projective subvarieties of positive dimension. Hence
the image of $\mathrm{Pic}^{c_2}(B)$ in $\mathfrak M_2$ is a point.
 However the double cover maps $C_E\to B$ differ.
\end{remark}

From Theorem 4.5 in \cite{Brinzanescu-Moraru2} 
it follows  that the set of all
bundles in $\mathcal{M}$ with fixed spectral cover $C$ is parametrised by 
$\mathrm{Prym}(C/B)$, where $\mathrm{Prym}(C/B)$ is here a smooth curve of genus 1.
 
Consider the spectral map 
\[H: \mathcal{M} \rightarrow \mathrm{Pic}^{c_2}(B)\]
that associates to each bundle $E$ in $\mathcal{M}$ its spectral cover $C_E$.
It is then an elliptic fibration whose fibre over $C$ in
${\rm Pic}^{c_2}(B)$ is $\mathrm{Prym}(C/B)$. 
Consequently, given that $\mathcal{M}$ is 
holomorphic symplectic, we have the following:
\begin{lemma}
The moduli space $\mathcal{M}$ is either an elliptic two-dimensional complex torus 
or a Kodaira surface.
\end{lemma}

We therefore know that $\mathcal{M}$ is an principal elliptic fibre bundle over $B$;
let us now show that it is in fact a principal $T^*$-bundle.

\begin{lemma}
\label{lemma:principal}
The moduli space $\mathcal{M}$ is a principal $T^*$-bundle over $B$.
\end{lemma}
\begin{proof}
It is sufficient to prove that if $C$ is the spectral cover of any 
element of $\mathcal{M}$, then $\mathrm{Prym}(C/B) \cong T^*$.  
Referring to \cite{ABT}, section 3, one can construct such a spectral cover as follows. 
Recall that, by corollary \ref{cor:smooth bisection}, the spectral curve of any bundle in $\mathcal{M}$  is smooth and irreducible. Recall also that any holomorphic rank-2 vector bundle on $X$ with Chern classes $c_1$ and $c_2$ is 
irreducible (see section 2).  

Consider the isogeny $\hat{c}_1: B \rightarrow T^*$ induced by $c_1$, 
which has odd degree $-c_1^2/2$ (see section \ref{sec:rank-2}).
The map $\psi := \hat{c}_1[2]: B[2] \rightarrow T^*[2]$ is then an isomorphism because $-c_1^2/2$ and $2$ are coprime.
Denote $H_{\psi} = \text{Graph}(\psi) \subset (B \times T^*)[2]$ and set $J_\psi := (B \times T^*)/H_{\psi}$. We must consider whether or not $\psi$ is irreducible.
If $\psi$ is reducible, then there exists a reducible rank-2 vector bundle on $X$ with Chern classes $c_1$ and $c_2$ (see \cite{ABT}, section 3, Case (b)), a contradiction. 
Consequently, $\psi$ an irreducible isometry and $J_\psi$ is the Jacobian of a curve $C$ of genus 2 that is a double cover of $B$ (see \cite{Kani}, section 1). Let us denote by $f: C \rightarrow B$ the double cover. Recall that $\mathrm{Prym}(C/B)$ is a connected component of the kernel of the norm map
$f_\ast : J_C = (B \times T^*)/H_{\psi} \rightarrow J_B \cong B$. However, referring to \cite{Kani}, 
section 1, $\ker{f_\ast} = J_T \cong T^*$, so that $\mathrm{Prym}(C/B) \cong T^*$.

Let $Y = X \times_B C$ and denote by $\varphi: Y \rightarrow X$ the double cover induced from $f: C \rightarrow B$. Then there exists a line bundle $L$ on $Y$ such that the holomorphic rank-2 vector bundle $E := \varphi_*(L)$ has Chern classes $c_1$ and $c_2$ (see \cite{ABT}, section 3, Case (a)). Moreover, the vector bundle $E$ has, by construction, spectral cover $C \rightarrow B$ \cite{Brinzanescu-Moraru2}. Recall from section 2 that $\mathbb{F}_{\det(E)} = \mathbb{F}_\delta$ since $c_1(\det(E)) = c_1(\delta)$. Consequently, by surjectivity of the spectral map $H:\mathcal{M} \rightarrow \mathrm{Pic}^{c_2}(B)$ \cite{Brinzanescu-Moraru3}, $C \rightarrow B$ is also the spectral cover of a vector bundle with determinant $\delta$ and second Chern class $c_2$.
In other words, $C \rightarrow B$ is the spectral cover of a bundle in $\mathcal{M}$ such that $\mathrm{Prym}(C/B) \cong T^*$.

\end{proof}

\begin{remark}
At this point we should mention the fact that the $T$-action on $X$ does not induce an action by pull-back on $\mathcal{M}$. The reason is that this action does not fix the determinant line bundle $\delta$ by the arguments of Lemma 4.8 in \cite{Te98}.  Indeed, since  $c_1(\delta)^2\neq 0$ (which occurs in our case because $\Delta(2,c_1,c_2)=\frac{1}{4}$), one proves that $\delta$ cannot
be fixed by the torsion part of $T$.
\end{remark}

\begin{theorem}
\label{thm:main}
The moduli space $\mathcal{M}$ is a primary Kodaira surface
with base $B$ and fibre $T^*$.
\end{theorem}

\proof
Suppose to the contrary that $\mathcal M$ is a 2-torus.
From Proposition \ref{prop:universal sheaf}, there exists an \'etale cover
$Y\to \mathcal M$ and a universal rank-2 vector bundle
$\mathcal E\to X\times Y$. 
Then that $Y$ is again a 2-torus; choose $\omega$ a 
K\"ahler form on $Y$.
There are three possible cases for the general restriction
$\mathcal E|_{x\times Y}$ of $\mathcal E$.

\medskip

{\em Case 1.}
For $x\in X$ general, $\mathcal E|_{x\times Y}$
is unstable with respect to $\omega$. Then there exists a relative destabilizing 
subsheaf $\mathcal F\subset\mathcal E$ of rank 1
(see also \cite{Paiola} for the existence of Harder-Narasimhan
filtrations in any rank).
We obtain a contradiction to the non-filtrability 
of $\mathcal E|_{X\times y}$ with $y\in Y$. Hence
this case cannot occur.

\medskip

{\em Case 2.}
For $x\in X$ general, $\mathcal E|_{x\times Y}$
is stable with respect to $\omega$. As in \cite{Teleman-CCM},
we denote by $X^{st}$, $X^{sst}$ the
Zariski open subsets of $X$ corresponding to stable, respectively semi-stable, 
vector bundles (sheaves) on $Y$, $B^{st}$, $B^{sst}$ the projections
to $B$, and by $\mathcal M^{st}$ the moduli
space of $\omega$-stable vector bundles on $Y$.
Theorem 1.5 loc.cit. shows that the induced map
$\varphi:X^{st}\to \mathcal M^{st}$ is either constant
or of generic rank one. If $\varphi$ is constant
then there exists a rank-2 vector bundle
$E$ on $Y$ such that $\mathcal E|_{X^{st}\times Y}\cong \varphi^*(E)$.
In particular, excepting finitely many points, for $b\in B$ we will have
$\mathcal E|_{T_b\times y}\cong \mathcal O^{\oplus 2}_{T_b}$ for any $y$.
This means that the associated spectral covers are double curves,
which contradicts Corollary \ref{cor:smooth bisection}.

If $\varphi$ is of rank one, from the proof of
\cite{Teleman-CCM}
it follows that $\varphi$ factors through $B^{st}$.
Indeed, the curves in $X^{st}$ found in the quoted proof
are restrictions of curves in $X$ which are of
the following type
\[
 D_E:=\{x\in X\ :\
h^0(Y,E^*\otimes \mathcal E|_{x\times Y})\ne 0\},
\]
where $E$ is a stable rank-2 vector bundle on $Y$.

Then for almost any $b\in B^{st}$ (we need to exclude
those points $b$ such that $T_b\not\subset X^{st}$), 
the induced map $\varphi_b:T_b\to \mathcal M^{st}$
is constant and hence there exists a rank-2
vector bundle $E_b$ on $Y$ such that $\mathcal E|_{T_b\times Y}\cong \varphi ^*(E_b)$.
As before the restriction $\mathcal E|_{T_b\times y}$ is trivial,
which is a contradiction.

\medskip

{\em Case 3.}
For $x\in X$ general, $\mathcal E|_{x\times Y}$ is properly semi-stable.
Note that, if $\mathcal E|_{x\times Y}$ is not decomposable then
its Jordan-H\"older filtration is unique, and there will
exist, as in the unstable case, a relative
Jordan-H\"older filtration which yields a contradiction 
of the filtrability. In conclusion, we can assume
that $\mathcal E|_{x\times Y}$ is decomposable 
$\mathcal E|_{x\times Y}\cong L_{1,x}\oplus L_{2,x}$ for a general $x\in X$
(note that $L_{1,x}$ and $L_{2,x}$ are defined at $x$, and are not global objects).
We reason along the lines of \cite{Teleman-CCM}, \cite{Toma-TEZA}.
Define, for $\eta\in \mathrm{NS}(Y)$, the Brill-Noether locus
\[
 X(\eta):=\{(x,\xi)\in X\times \mathrm{Pic}^\eta(Y)\ :\
h^0(Y,\mathcal P_\xi^*\otimes \mathcal E|_{x\times Y})\ne 0\}
\]
where $\mathcal P$ is the Poincar\'e bundle on $Y\times \mathrm{Pic}^\eta(Y)$.

The projections of all $X(\eta)$, with $\eta\in \mathrm{NS}(Y)$, on $X$ are Zariski closed,
and their union coincides with $X$, by assumption.
Since $\mathrm{NS}(Y)$ is countable, there is a class $\eta\in \mathrm{NS}(Y)$ such that
$X(\eta)$ covers $X$. Let us fix one such a class $\eta$
and an irreducible component $X'$ of $X(\eta)$ which projects onto $X$
and consider $\widetilde{X}$ a desingularization of $X'$.
We may also assume that $\mu_\omega(\eta)=\mu_\omega(\mathcal E|_{x\times Y})$.

If $h^0(Y,\mathcal P_\xi^*\otimes \mathcal E|_{x\times Y})=1$
for a general $(x,\xi)\in X'$ then, as before we 
shall reach a contradiction of
the non-filtrability of the restrictions of $\mathcal E$.
Indeed, this situation represents the case
$L_{1,x}\not\cong L_{2,x}$. Thus the map
$X'\to X$ is generically finite with one or two sheets,
and hence $\widetilde{X}$ is a non-K\"ahler surface.
It is clear that the induced projection
$\widetilde{X}\to\mathrm{Pic}^\eta(Y)$ 
is either constant or of generic rank one.
If it is constant, we get once more a contradiction
of the non-filtrability of $\mathcal E|_{X\times y}$.
If it is of generic rank one, then let $F$ be one 
component of one of its (general) fibres which covers a fibre $T_0$ of
$\pi:X\to B$. Let $\xi$ be the image of $F$
in $\mathrm{Pic}^\eta(Y)$. We see that
$h^0(\mathcal E|_{x\times Y}\otimes \mathcal P_\xi^*)=1$
for $x\in T_0$ general. Hence the direct image sheaf $M$
of $\mathcal E|_{T_0\times Y}\otimes \mathcal P_\xi^*$
through the projection $p_0$ on $T_0$ 
is of rank one, which implies the existence
of a non-trivial morphism
$p_0^*(M)\otimes \mathcal P_\xi\to \mathcal E|_{T_0\times Y}$.
In particular, all the spectral curves in $B\times T^*$
of the bundles $\mathcal E|_{X\times y}$ pass through 
one fixed point of $T_0^*$ which is impossible;
since the determinant of the bundles $\mathcal E|_{X\times y}$
is fixed, the second point in $T_0^*$ must be fixed, too.
This contradicts Remark \ref{rmk: splitting type}.

The case $h^0(Y,\mathcal P_\xi^*\otimes \mathcal E|_{x\times Y})\ge 2$
corresponds precisely to the situation where $\mathcal E|_{x\times Y}\cong
\mathcal P_\xi^{\oplus 2}$ for some $\xi\in\mathrm{Pic}^\eta(Y)$;
in fact, $h^0(Y,\mathcal P_\xi^*\otimes \mathcal E|_{x\times Y}) = 2$. 
In particular, we obtain a map $x\mapsto \xi$, from $\widetilde{X}$ to
$\mathrm{Pic}^\eta(Y)$. As before, we have two cases. If this map
is constant, then we infer that $\mathcal E$ is a pull-back from $X$.
This is impossible, as $Y$ is a component of the moduli space.
If this map is of generic rank one, then, for a general fibre $T$ of $\pi$,
$\mathcal E|_{T\times Y}$ is the pullback of a sheaf on $T$, which fixes again
the two points in $T^*$ for all spectral curves.
We obtain a contradiction of the Remark~\ref{rmk: splitting type}.
\endproof

\begin{remark}
 The original Kodaira surface $X$ and its moduli pair $\mathcal M$ have the
 same Neron-Severi group over $\mathbb Q$, see \cite{BrinzanescuLNM}
 for the description of Neron-Severi groups of elliptic fibrations.
\end{remark}

\begin{remark}
 From the construction, we have a universal 
 curve $\mathcal C\subset B\times T^*\times B^*$, and
 denote by $\mathcal W\subset X\times T^*\times B^*$ the induced
 fibered product.
 It is clear that we have a natural section $\Sigma_{\mathcal C}:\mathcal C\to  
 \mathcal C\times T^*=\mathrm{Jac}(\mathcal W/\mathcal C)$
 induced by the projection $B\times T^*\times B^*\to T^*$. 
 However, this section is not the spectral curve of a line bundle $\mathcal L$
 on $\mathcal W$. Indeed, if it were, then the push-forward of $\mathcal L$
 to $X\times B^*$ would produce a section (possibly a
 multi-section)
 $B^*\to \mathcal M$, which would contradict Theorem \ref{thm:main}.
 This fact is an obstruction to the realization of the universal
 rank-two bundle on $X\times \mathcal M$ (supposing that
 it exists) as the push-forward of some universal line bundle
 on an appropriate double cover.
\end{remark}

\section{Topologically isomorphic pairs}
\label{sec:pairs}

We prove next that, in the presence of a universal bundle (for instance,
in the setup of Example \ref{ex:universal}), the 
original primary Kodaira surface and its moduli pair are
topologically isomorphic up to some finite unramified cover. 
For the existence of a homeomorphism we use
the following topological fact.

\begin{lemma}
\label{lemma:iso}
 Two primary Kodaira surfaces $X$ and $X'$ 
are topologically isomorphic if and only if 
$\mathrm{ord}(Tors(\mathrm{NS}(X))=\mathrm{ord}(Tors(\mathrm{NS}(X'))$.
\end{lemma}

\proof
The proof is implicitly contained in \cite{Ko}, 
\cite{BHPV}, \cite{BrinzanescuLNM}. 

Note that the torsion of the Neron-Severi group of a Kodaira surface $X$ 
coincides with the torsion of $H^2(X,\mathbb Z)$.
In his foundational paper \cite{Ko}, Kodaira describes the
fundamental group and the first homology group of a
Kodaira surface $X$. It  follows that the topological
type is determined by a positive integer $m$. On the
other hand, $H_1(X,\mathbb Z)\cong \mathbb Z^{\oplus 3}\oplus \mathbb Z/m\mathbb Z$.
Our assertion is an immediate consequence of the universal coefficient
formula.
\endproof

\begin{lemma}
\label{lemma:unramified}
 Let $X$ and $X^\prime$ be two primary Kodaira surfaces, and
 denote $m=\mathrm{ord}(Tors(\mathrm{NS}(X))$
 and $m^\prime=\mathrm{ord}(Tors(\mathrm{NS}(X'))$.
 If $m | m^\prime$, then $X^\prime$ is homeomorphic to
 an unramified cover of $X$.
\end{lemma}

\proof
Let $T=\mathbb C/\Gamma$ be the fibre
of the fibration $X\to B$, and denote by $c\in H^2(B,\Gamma)$ 
the topological class of the principal fibre bundle $X\to B$, \cite{BHPV} p. 195. 
There is a primitive embedding $i:\mathbb Z\to \Gamma$ such that
$c$ is the image of an element $\eta\in H^2(B,\mathbb Z)$ through 
$H^2(i):H^2(B,\mathbb Z)\to H^2(B,\Gamma)$ loc. cit.
Taking continuous sections in the diagram 
\xycenter{
0\ar[r] & \mathbb Z \ar[r]\ar[d]^{i} & \mathbb C \ar[r]\ar@{=}[d] & \mathbb C^*\ar[r]\ar[d] & 0 \\
0\ar[r] & \Gamma \ar[r] & \mathbb C \ar[r] & T\ar[r] & 0
}
it follows, as in loc. cit. V Prop. 5.2, that $X$ is topologically a product of $S^1$
with the $S^1$-bundle over $B$ given by $\eta$. By our previous considerations
it becomes clear that $\eta$ may be identified to $m\in \mathbb Z$
via the canonical isomorphism $H^2(B,\mathbb Z)\cong \mathbb Z$. The conclusion now follows.
\endproof

With the preceeding notation, for text
coherence put $X':=\mathcal M$.

\begin{theorem}
\label{thm:iso}
 If a universal rank-two
vector bundle on $X\times X'$ exists, then
$X'$ is topologically isomorphic with an \'etale cover of $X$. 
\end{theorem}

In the proof, we shall need the following general fact:

\begin{lemma}
\label{lem:double cover}
Let $X$, $Y$ be two complex manifolds, not necessarily compact,
together with a $(2:1)$ holomorphic map $f:Y\to X$, and denote by $\iota:Y\to Y$
the natural involution and by $R$ the ramification divisor on $Y$.
Then, for any line bundle $\mathcal L$ on $Y$, we have a
short exact sequence
\[
0\to f^*f_*\mathcal L\stackrel{\alpha}{\longrightarrow} 
\mathcal L\oplus \iota^*\mathcal L\stackrel{\beta}{\longrightarrow} \mathcal L|_R\to 0.
\] 
\end{lemma}

\proof
Since $f_*\iota^*\mathcal L=f_*\iota_*\iota^*\mathcal L=f_*\mathcal L$, by the base-change formula, and
$\mathcal L|_R\cong \iota^*\mathcal L|_R$, 
the morphisms $\alpha$, $\beta$ from the statement are naturally defined.
On the other hand, exactness is a local property, and hence
we can assume that $\mathcal L\cong \mathcal O_Y$.
For any point $y\in Y$, the induced short sequence
on stalks becomes
\[
0\to (f_*\mathcal O_Y)_{f(y)}\otimes_{\mathcal O_{X,f(y)}}\mathcal O_{Y,y}
\stackrel{\alpha_y}{\longrightarrow} \mathcal O_{Y,y}\oplus\mathcal O_{Y,\iota(y)}
\stackrel{\beta_y}{\longrightarrow}\mathcal O_{R,y}\to 0,
\]
where $\alpha_y:s\otimes\sigma\mapsto(s\sigma,s\iota^*\sigma)$, 
and $\beta_y:(s_1,s_2)\mapsto
(s_1|_R-s_2|_R)$. Obviously, if $y\not\in R$, then
$\mathcal O_{R,y}=0$, and $\alpha_y$ is an isomorphism.
moreover, for any $y$, $\alpha_y$ is injective, $\beta_y$ 
is surjective, and $\beta_y\circ \alpha_y=0$.
The exactness in the middle follows from Nakayama's lemma
and from the Grauert theorem, passing to the fibres of
vector bundles for the morphism $\alpha$.
\endproof

\proof(of Theorem \ref{thm:iso})
Let $m$ and $m'$ be the orders of the torsion
parts of $\mathrm{NS}(X)$ and $\mathrm{NS}(X')$ respectively.
Consider $\mathcal E\to X\times X'$ the universal
rank-two vector bundle; it is not uniquely defined, see below. 
The proof strategy is to apply  the Fourier-Mukai
transform and the induced determinant map,
to the structure sheaves of fibres
of $X\to B$, $X'\to B$. More precisely,
let $K(X)$ and $K(X')$ be
the Grothendieck groups, and 
consider $\lambda_{\mathcal E}:K(X')\to \mathrm{Pic}(X)$ 
the composition \cite{HL}, p. 214
\[
 K(X')\stackrel{FM_{\mathcal E}}{\longrightarrow}
K(X)\stackrel{\mathrm{det}}{\longrightarrow}\mathrm{Pic}(X),
\]
where $FM_\mathcal E$ denotes the Fourier-Mukai transform
with respect to $\mathcal E$, and $\varphi_{\mathcal E}
:=c_1\circ \lambda_{\mathcal E}$. 
Note that the restriction of $\lambda_\mathcal E$ to
$\mathrm{Pic}(X')$ is continuous, as is easily seen using Grauert's base change theorem, \cite{BS} Thm. 3.3.4.
Let $F'\subset X'$ 
be an elliptic fibre. Since $\mathcal O_{X'}$ and $\mathcal O_{X'}(-m'F')$
both belong to $\mathrm{Pic}^0(X')$, it follows
that $\varphi_{\mathcal E}([\mathcal O_{m'F'}])=0$.
On the other side, we will show that $\varphi_{\mathcal E}(\mathcal O_{F'})=
c_1(\mathcal O_{F})$, where $F$ is a fibre of $X\to B$;
note that $[\mathcal O_{kF'}]=k\cdot [\mathcal O_{F'}]$ and
$[\mathcal O_{\sum_i{F_i'}}]=\sum_i[\mathcal O_{F_i'}]$
in the Grothendieck group of $X'$.
This eventually implies that $m\ |\ m'$ and
we apply Lemma \ref{lemma:unramified}. 

\medskip

We prove that $\varphi_{\mathcal E}([\mathcal O_{F'}])=
c_1(\mathcal O_{F})$.
Let $b\in B$ be a point. It corresponds to
a spectral cover $f:C\to B$. 
Let $X'_b\cong\mathrm{Prym}(C/B)\cong T^*$ be the 
fibre of $X'\to B$ and set
$\mathcal E_b:=\mathcal E|_{X\times X'_b}$.

\medskip

{\em Claim.} The Poincar\'e bundle on $C\times \mathrm{Jac}(C)$
induces by restriction a universal line bundle $\mathcal P$
on $C\times \mathrm{Prym}(C/B)$. This universal bundle is
uniquely determined by the choice of a point $c\in C$, modulo
the involution $\iota:C\to C$ in the following way:
$\mathcal P|_{c\times \mathrm{Prym}(C/B)}\cong \mathcal O$
and $\mathcal P|_{\iota(c)\times \mathrm{Prym}(C/B)}\cong \mathcal O$.
The choice of another point $c'\in C$, $c'\not\in\{c,\iota(c)\}$
will change $\mathcal P$ by a twist with a pullback
of a line bundle on $\mathrm{Prym}(C/B)$.

\medskip

Consider now the cartesian diagram
\xycenter{
X_C \times X'_b \ar[r]^{f_{X\times X'_b}} \ar[d]^{\mathrm{pr}_{X_C}} & 
X\times X'_b \ar[d]^{\mathrm{pr}_{X}} \\
X_C \ar[r]^{f_X} \ar[d]^{\pi_C} & X \ar[d]^{\pi} \\
C \ar[r]^{f}  & B.
}
Then, by base-change 
\[
 f^*_X\left(R^\bullet\mathrm{pr}_{X,*}\mathcal E_b\right)
\cong R^\bullet\mathrm{pr}_{X_C,*}\left(f^*_{X\times X'_b}\mathcal E_b\right).
\]

Following Theorem 4.5 \cite{Brinzanescu-Moraru2},
$\mathcal E_b\cong 
f_{X\times X'_b,\ *}(\mathrm{pr}_{X_C}^*(L)\otimes \widetilde{\mathcal P})$
where $L\in\mathrm{Pic}(X_C)$ corresponds to $\Sigma_C$ and $\widetilde{\mathcal P}$
is the inverse image of a suitable
universal bundle $\mathcal P$ on ${C\times \mathrm{Prym}(C/B)}$
on $X_C\times X'_b$. The bundle $\mathcal P$ depends, as
mentioned above, on a trivialization along a particular fibre $\{c,\iota(c)\}$ of $f$.
This is imposed by our choice of the universal bundle $\mathcal E$.
Note indeed that we may modify $\mathcal E$ twisting with a pull-back
of a line bundle on $X'$, and $X'_b\cong \mathrm{Prym}(C/B)$.
Note also that we can arrange $c$ to be outside the ramification locus,
by changing another point $b$ if necessary.

Put
\[
 \mathcal F_b:=f_{X\times X'_b}^*(\mathcal E_b)
=f_{X\times X'_b}^*(f_{X\times X'_b,\ *}(\mathrm{pr}_{X_C}^*(L)
\otimes \widetilde{\mathcal P})).
\]

From Lemma \ref{lem:double cover},
$\mathcal F_b$ is presented as an elementary modification
\begin{equation}
\label{eqn:modif}
  0\to \mathcal F_b\to(\mathrm{pr}_{X_C}^*(L)\otimes \widetilde{\mathcal P})
\oplus \widetilde{\iota}^*(\mathrm{pr}_{X_C}^*(L)\otimes \widetilde{\mathcal P})
\to (\mathrm{pr}_{X_C}^*(L)\otimes \widetilde{\mathcal P})|_R\to 0,
\end{equation}
where $\widetilde{\iota}$ is the induced involution on $X_C\times X'_b$ associated to the
double cover, and $R$ is the ramification divisor. 

Note that the restriction of
$\mathrm{pr}_{X_C}^*(L)\otimes \widetilde{\mathcal P}$ 
to the general fibre of $\mathrm{pr}_{X_C}$
has no global sections,  hence
$R^0\mathrm{pr}_{X_C,*}(\mathrm{pr}_{X_C}^*(L)\otimes \widetilde{\mathcal P})
=L\otimes R^0\mathrm{pr}_{X_C,*}(\widetilde{\mathcal P})=0$ as it is torsion-free.
This implies that $R^0\mathrm{pr}_{X_C,*}(\mathcal F_b)=0$.
However, the restriction of $\mathrm{pr}_{X_C}^*(L)\otimes \widetilde{\mathcal P}$
to the fibre $x\times X'_b$ of $\mathrm{pr}_{X_C}$ over any point $x\in \pi_C^{-1}(\{c,\iota(c)\})$  
will have 1-dimensional space of global sections. By the
Riemann-Roch Theorem applied on the fibres, and by the
semi-continuity Theorem, it follows that 
$R^1\mathrm{pr}_{X_C,*}(\mathcal F_b)$ is torsion,
supported on $\pi_C^{-1}(\{c,\iota(c)\})$, and is of rank one
along $\pi_C^{-1}(\{c,\iota(c)\})$. Then 
$R^1\mathrm{pr}_{X,*}(\mathcal E_b)$ is supported
on the fibre $F:=X_{f(c)}$ of $X\to B$ over $f(c)$, and is
of rank one along this fibre. In particular, 
$c_1(R^1\mathrm{pr}_{X,*}(\mathcal E_b))=c_1(\mathcal O_F)\in \mathrm{NS}(X)$.
Note that the same is true even if $c$ was a ramification point
for $f$; analyze the sequence (\ref{eqn:modif}).
We have proved that $m\ |\ m'$, which concludes the proof.
\endproof

\begin{remark}
 The same statement remains true if we replace $X'$ 
by an \'etale cover $X''\to X'$, induced by some
cover of $B$, and $\mathcal E$ by its pullback. 
In fact, in the proof we do not use the unicity of
representatives in the family.
\end{remark}

\end{document}